\documentclass[12pt,a4paper]{article}

\usepackage{amssymb,amsmath,amsthm}

\usepackage[colorlinks=true,citecolor=black,linkcolor=black,urlcolor=blue]{hyperref}
\newcommand{\arxiv}[1]{\href{http://arxiv.org/abs/#1}{\texttt{arXiv:#1}}}

\theoremstyle{definition}
\newtheorem{defn}{Definition}[section]
\newtheorem{ex}[defn]{Example}
\newtheorem{rem}[defn]{Remark}

\theoremstyle{plain}
\newtheorem{lem}[defn]{Lemma}
\newtheorem{thm}[defn]{Theorem}
\newtheorem{prop}[defn]{Proposition}

\numberwithin{equation}{section}

\allowdisplaybreaks[4]

\newlength{\BiblioSpacing}
\renewenvironment{thebibliography}[1]{%
\begin{oldthebibliography}{#1}%
\setlength{\parskip}{\BiblioSpacing}
\setlength{\itemsep}{\BiblioSpacing}
}%
{%
\end{oldthebibliography}%
}

\title{On relative $t$-designs in\\ polynomial association schemes}
\author{Eiichi Bannai\\
\small Department of Mathematics\\[-0.8ex]
\small Shanghai Jiao Tong University\\[-0.8ex] 
\small Shanghai 200240, China\\
\small\tt bannai@sjtu.edu.cn\\
\and
Etsuko Bannai\\
\small Misakigaoka 2-8-21, Itoshima 819-1136, Japan\\
\small\tt et-ban@rc4.so-net.ne.jp\\
\and
Sho Suda\\
\small Department of Mathematics Education\\[-0.8ex]
\small Aichi University of Education\\[-0.8ex] 
\small Kariya 448-8542, Japan\\
\small\tt suda@auecc.aichi-edu.ac.jp\\
\and
Hajime Tanaka\\
\small Research Center for\\[-0.8ex]
\small Pure and Applied Mathematics\\[-0.8ex]
\small Graduate School of Information Sciences\\[-0.8ex]
\small Tohoku University\\[-0.8ex] 
\small Sendai 980-8579, Japan\\
\small\tt htanaka@m.tohoku.ac.jp\\
}
\date{
\small Mathematics Subject Classifications: 05E30, 05B30}

\begin{document}
\maketitle

\begin{abstract}
Motivated by the similarities between the theory of spherical $t$-designs and that of $t$-designs in $Q$-polynomial association schemes, we study two versions of \emph{relative $t$-designs}, the counterparts of Euclidean $t$-designs for $P$- and/or $Q$-polynomial association schemes.
We develop the theory based on the \emph{Terwilliger algebra}, which is a noncommutative associative semisimple $\mathbb{C}$-algebra associated with each vertex of an association scheme.
We compute explicitly the Fisher type lower bounds on the sizes of relative $t$-designs, assuming that certain irreducible modules behave nicely.
The two versions of relative $t$-designs turn out to be equivalent in the case of the Hamming schemes.
From this point of view, we establish a new algebraic characterization of the Hamming schemes.

\bigskip\noindent \textbf{Keywords:} Relative $t$-design; Fisher type inequality; Terwilliger algebra
\end{abstract}

\section{Introduction}

Design theory is concerned with finding ``good'' finite sets that ``approximate globally'' their underlying spaces (often) having strong symmetry/regularity, such as the Euclidean space $\mathbb{R}^n$, the unit sphere $S^{n-1}\subseteq\mathbb{R}^n$, and the set of $k$-subsets of a given $v$-set.
It has therefore a vast range of applications in various fields of science.
See, e.g., \cite{CD2007B,BB2009EJC}.

The similarities between the theories of spherical $t$-designs and combinatorial $t$-$(v,k,\lambda)$ designs are well known; cf.~\cite{DGS1977GD,Delsarte1978SIAM,Godsil1993B,BB1999B}.
Historically, the concept of spherical $t$-designs was introduced by Delsarte, Goethals, and Seidel \cite{DGS1977GD} as a continuous analogue of that of $t$-designs in $Q$-polynomial association schemes due to Delsarte \cite{Delsarte1973PRRS,Delsarte1976JCTA}.
(Combinatorial $t$-$(v,k,\lambda)$ designs are precisely the $t$-designs in the Johnson scheme $J(v,k)$.)
It was then generalized to the concept of Euclidean $t$-designs by Neumaier and Seidel \cite{NS1988IM}, and Euclidean $t$-designs quickly became an active area of research; cf.~\cite{BB2009EJC}.
Although the counterparts of Euclidean $t$-designs in the theory of $Q$-polynomial association schemes were already defined and discussed to some extent by Delsarte \cite{Delsarte1977PRR} (cf.~\cite{BB2012JAMC}) much earlier as \emph{relative $t$-designs}, it seems that the theory of the latter has not been fully developed yet (except in the case of the binary Hamming scheme $H(n,2)$, in which case relative $t$-designs turn out to be equivalent to \emph{regular $t$-wise balanced designs}).
This paper is a contribution to this theory.
Our discussions also include a concept of relative $t$-designs in general $P$-polynomial association schemes as well, following Delsarte and Seidel \cite{DS1989LAA}.

We refer the reader to \cite{Delsarte1973PRRS,BI1984B,BCN1989B,Godsil1993B,MT2009EJC,DKT2014pre}, etc., for the background on association schemes and some fundamental concepts.
Throughout the paper, let $\mathfrak{X}=(X,\{R_r\}_{r=0}^d)$ be a (symmetric) $d$-class association scheme, and fix a base vertex $u_0\in X$.
Let $X_r=\{x\in X\, |\, (u_0,x)\in R_{r}\}$ for $r=0,1,\dots,d$.
We call $X_0,X_1,\dots,X_d$ the \emph{shells} of $\mathfrak{X}$.
Let $\mathcal{F}(X)$ be the vector space consisting of all the real valued functions on $X$.
In the following arguments we often identify $\mathcal{F}(X)$ with the vector space $\mathbb R^{X}$ consisting of the real column vectors with coordinates indexed by $X$.

We first introduce a concept of $t$-designs for general $P$-polynomial association schemes.
Suppose that $\mathfrak{X}$ is $P$-polynomial with respect to the ordering $R_0,R_1,\dots,R_d$.
In the study of spherical/Euclidean $t$-designs in $\mathbb{R}^n$, we work with the vector space of polynomials in $n$ variables, in particular with the subspaces of homogeneous polynomials.
For the $P$-polynomial scheme $\mathfrak{X}$, it is natural to consider the following subspaces of $\mathcal{F}(X)$.
For every $z\in X_j$, we define $f_z\in \mathcal{F}(X)$ by
\begin{equation*}
	f_z(x)=\begin{cases}	1, & \text{if} \ x\in X_i, \ i\geq j, \ \text{and} \ (x,z)\in R_{i-j}, \\ 0, & \text{otherwise}, \end{cases} \qquad (x\in X).
\end{equation*}
In other words, $f_z(x)=1$ if and only if $z$ lies on a geodesic between $u_0$ and $x$ in the corresponding distance-regular graph $(X,R_1)$.
Let $\operatorname{Hom}_j(X)=\operatorname{span}\{ f_z\, |\, z\in X_j\}$ $(j=0,1,\dots,d)$.
Then,
\begin{equation*}
	\dim (\operatorname{Hom}_j(X))=|X_j|=:k_j \quad (j=0,1,\dots,d),
\end{equation*}
and we have the following direct sum decomposition of $\mathcal{F}(X)$:
\begin{equation*}
	\mathcal{F}(X)=\operatorname{Hom}_0(X) + \operatorname{Hom}_1(X) + \dots + \operatorname{Hom}_d(X).
\end{equation*}

We now consider a (positive) weighted subset $(Y,w)$ of $X$, that is to say, a pair of a subset $Y$ of $X$ and a function $w:Y\rightarrow (0,\infty)$.
Let $\{r_1,r_2,\dots,r_p\}=\{r\,|\, Y\cap X_r\ne\emptyset\}$, and let $Y_{r_i}=Y\cap X_{r_i}$, $w(Y_{r_i})=\sum_{y\in Y_{r_i}}w(y)$ for $i=1,2,\dots,p$.
We say that $Y$ is \emph{supported} by the union $S=X_{r_1}\cup X_{r_2}\cup\dots\cup X_{r_p}$ of $p$ shells.
For any subspace $R(X)$ of $\mathcal{F}(X)$, we write $R(S)=\{ f|_S\, |\, f\in R(X)\}$.

\begin{defn}[$P$-polynomial case]\label{design-P}
A weighted subset $(Y, w)$ of $X$ is a \emph{relative $t$-design of $\mathfrak{X}$ with respect to $u_0$} if  
\begin{equation*}
	\sum_{i=1}^p\frac{w(Y_{r_i})}{k_{r_i}} \sum_{x\in X_{r_i}}f(x)=\sum_{y\in Y}w(y)f(y)
\end{equation*}
for every $f\in \operatorname{Hom}_0(X)+\operatorname{Hom}_1(X)+\dots+\operatorname{Hom}_t(X)$.
\end{defn}

\noindent
This definition is due to Delsarte and Seidel \cite[Section 6]{DS1989LAA} for the binary Hamming scheme $H(d,2)$.
In this paper, we mostly consider the case $t=2e$ for simplicity.

\begin{thm}[{\cite{DS1989LAA}}]\label{Fisher-P}
Let $(Y, w)$ be a relative $2e$-design of a Hamming scheme $H(d,q)$ with respect to $u_0$ in the sense of Definition \ref{design-P}.
Let $S=X_{r_1}\cup\dots\cup X_{r_p}$ be the union of the shells which support $Y$.
Then,
\begin{equation}\label{Fisher-P-Hom}
	|Y|\geq \dim\big(\!\operatorname{Hom}_{0}(S)+\operatorname{Hom}_{1}(S)+\dots+\operatorname{Hom}_{e}(S)\big).
\end{equation}
\end{thm}

\noindent
Delsarte and Seidel \cite{DS1989LAA} proved Theorem \ref{Fisher-P} only for $H(d,2)$, but their proof works for general $q$.
Theorem \ref{Fisher-P} also follows from Theorem \ref{Fisher-Q} and Proposition \ref{Hom-L} below.
Recently, Xiang \cite{Xiang2012JCTA} succeeded in determining the right hand side of \eqref{Fisher-P-Hom} explicitly for $H(d,2)$, which was left open in \cite{DS1989LAA}.
Namely, he proved 
\begin{equation}\label{Fisher-P-exact}
	\dim\big(\!\operatorname{Hom}_{0}(S)+\operatorname{Hom}_{1}(S)+\dots+\operatorname{Hom}_{e}(S)\big) =k_e+k_{e-1}+\dots + k_{e-p+1},
\end{equation}
under a reasonable additional condition which avoids the triviality.
In this paper, we focus on generalizing \eqref{Fisher-P-exact} to other classes of $P$-polynomial association schemes (without necessarily reference to Theorem \ref{Fisher-P} itself).
In Appendix \ref{sec: dual-polar}, we do, however, show that Theorem \ref{Fisher-P} is valid for dual polar schemes as well.

The concept of relative $t$-designs for $Q$-polynomial association schemes was introduced by Delsarte \cite{Delsarte1977PRR}.
We now recall the definition.
Suppose that $\mathfrak{X}$ is $Q$-polynomial with respect to the ordering $E_0,E_1,\dots,E_d$ of its primitive idempotents, and let $L_j(X)(\subseteq \mathcal{F}(X))$ be the column space of $E_j$ ($j=0,1,\dots,d$).
Then,
\begin{equation*}
	\dim (L_j(X))=\operatorname{rank}(E_j)=:m_j \quad (j=0,1,\dots,d),
\end{equation*}
and we have the following orthogonal direct sum decomposition of $\mathcal{F}(X)$:
\begin{equation*}
	\mathcal{F}(X)=L_0(X)\perp L_1(X)\perp \dots\perp L_d(X).
\end{equation*}

\begin{defn}[$Q$-polynomial case]\label{design-Q}
A weighted subset $(Y,w)$ of $X$ is \emph{a relative $t$-design of $\mathfrak{X}$ with respect to $u_0$} if
\begin{equation*}
	\sum_{i=1}^p \frac{w(Y_{r_i})}{k_{r_i}} \sum_{x\in X_{r_i}} f(x) =\sum_{y\in Y}w(y)f(y)
\end{equation*}
for every $f\in L_0(X)\perp L_1(X)\perp \dots\perp L_{t}(X)$.
\end{defn}

\noindent
Bannai and Bannai \cite{BB2012JAMC} obtained the following Fisher type inequality for general $Q$-polynomial association schemes:

\begin{thm}[{\cite{BB2012JAMC}}]\label{Fisher-Q}
Let $(Y,w)$ be a relative $2e$-design of the $Q$-polynomial scheme $\mathfrak{X}$ with respect to $u_0$ in the sense of Definition \ref{design-Q}.
Let $S=X_{r_1}\cup\dots\cup X_{r_p}$ be the union of the shells which support $Y$.
Then,
\begin{equation}\label{Fisher-Q-L}
	|Y|\geq \dim\big(L_0(S)+L_1(S)+\dots+L_e(S)\big).
\end{equation}
\end{thm}

\noindent
As in the case of \eqref{Fisher-P-Hom}, it was not easy to compute the right hand side of \eqref{Fisher-Q-L} explicitly.
The initial attempt was made by Li, Bannai, and Bannai \cite{LBB2013GC} for $H(d,2)$, but was unsuccessful in general. 
Then, this attempt lead Xiang to obtain a successful result in the general case for $H(d,2)$, as it is known that the two definitions of relative $t$-designs are essentially equivalent for $H(d,2)$.
Namely, both definitions are shown to be equivalent to the geometric definition of relative $t$-designs coming from the structure of the regular semilattice associated with $H(d,2)$; cf.~\cite{Delsarte1977PRR}.
The equivalence of Definition \ref{design-P} for $H(d,2)$ with the definition of regular $t$-wise balanced designs was pointed out by Delsarte and Seidel \cite[Theorem 6.2]{DS1989LAA}, whereas the equivalence of Definition \ref{design-Q} for $H(d,2)$ with the geometric definition of relative $t$-designs was established by Delsarte \cite[Theorem 9.8]{Delsarte1977PRR} (see also \cite{BBI2013B}).
However, we note that

\begin{prop}\label{Hom-L}
If $\mathfrak{X}$ is a Hamming scheme $H(d,q)$, then for $t=0,1,\dots,d$,
\begin{equation}\label{Hom_e-L_e}
	\operatorname{Hom}_0(X)+\operatorname{Hom}_1(X)+\dots +\operatorname{Hom}_t(X)=L_0(X)+L_1(X)+\dots +L_t(X).
\end{equation}
\end{prop} 

\begin{proof}
Without loss of generality, we may suppose that $X=\{0,1,\dots,q-1\}^d$ and $u_0=(0,0,\dots,0)$.
Let $z=(z_1,z_2,\dots,z_d)\in X_j$.
Note that $z$ has exactly $j$ nonzero entries, and let $\ell_1,\ell_2,\dots,\ell_j$ be the corresponding coordinates.
Then, it is easy to see that $f_z$ is the characteristic function of the subset $\{(x_1,x_2,\dots,x_d)\in X\, |\, x_{\ell_h}=z_{\ell_h} \ (h=1,2,\dots,j) \}$, which is known to be contained in $\sum_{i=0}^jL_i(X)$; see, e.g., \cite{Delsarte1976JCTA,Stanton1985JCTA}.\footnote{In Appendix \ref{sec: Hom-L}, we give a direct proof that $f_z$ belongs to $\sum_{i=0}^jL_i(X)$, which does not use the theory of regular semilattices found in \cite{Delsarte1976JCTA,Stanton1985JCTA}.}
Since both sides of \eqref{Hom_e-L_e} have the same dimension, we obtain the desired result.
\end{proof}

\noindent
Thus, for $H(d,q)$, relative $t$-designs in the sense of Definition \ref{design-P} are equivalent to relative $t$-designs in the sense of Definition \ref{design-Q}.
This observation seems to be new for $H(d,q)$ for general $q$.
As is mentioned before, for $H(d,2)$, the result of Xiang \cite{Xiang2012JCTA} implies that the right hand side of \eqref{Fisher-Q-L} is also given explicitly by
\begin{equation}\label{LB-Q}
	\dim\big(L_0(S)+L_1(S)+\dots+L_e(S)\big) = m_e+m_{e-1}+\dots +m_{e-p+1},
\end{equation}
since $m_j=k_j$ $(j=0,1,\dots,d)$ in this case.
In a private communication, Xiang extended his main result in \cite{Xiang2012JCTA} to general $q$.
Thus, the right hand side of \eqref{Fisher-Q-L} is also given explicitly as \eqref{LB-Q} for $H(d,q)$.

In this paper, we investigate to what extent the above results can be generalized to other $P$- and/or $Q$-polynomial association schemes.
In Section \ref{sec: exact Fisher}, we derive sufficient conditions that \eqref{Fisher-P-exact} (resp.~\eqref{LB-Q}) holds for a $P$-polynomial (resp.~$Q$-polynomial) association scheme (Theorems \ref{criterion-P} and \ref{criterion-Q}).
These conditions can be readily checked for $H(d,q)$, so that we obtain different proofs of the results of Xiang mentioned above.
Concerning \eqref{Hom_e-L_e}, we first suspected that a similar result might hold for general (formally) self-dual $P$- and $Q$-polynomial association schemes, but it turns out that this is not the case in general.
Indeed, in Section \ref{sec: Hamming}, we show that if $\mathfrak{X}$ is formally self-dual, $P$-polynomial (and thus $Q$-polynomial), and satisfies $\operatorname{Hom}_0(X) + \operatorname{Hom}_1(X)= L_0(X) + L_1(X)$, then $\mathfrak{X}$ must be a Hamming scheme $H(d,q)$, provided that $d\geq 6$ (Theorem \ref{Hamming}).
All of these theorems are proved using the theory of the \emph{Terwilliger algebra} \cite{Terwilliger1992JAC,Terwilliger1993JACa,Terwilliger1993JACb}.
See \cite{Tanaka2009EJC} for more applications of the Terwilliger algebra to design theory.

\section{Computations of the Fisher type lower bounds}\label{sec: exact Fisher}

In this section and the next, we shall use some basic facts about the Terwilliger algebra.
In this context, we shall work with the complex vector space $\mathbb{C}^X$ instead of $\mathbb{R}^X$, but we note that the dimensions of the various subspaces in question do not change, as they are spanned by real vectors.

We use the following notation.
For every $x\in X$, let $\hat{x}\in \mathcal{F}(X)=\mathbb{C}^X$ be the characteristic function of the set $\{x\}$.
Let $A_0,A_1,\dots,A_d$ and $E_0,E_1,\dots,E_d$ be (fixed orderings of) the adjacency matrices and the primitive idempotents of $\mathfrak{X}$, respectively.
Let $E_0^*,E_1^*,\dots,E_d^*$ and $A_0^*,A_1^*,\dots,A_d^*$ be the diagonal matrices with diagonal entries $(E_i^*)_{xx}=(A_i)_{u_0x}$ and $(A_i^*)_{xx}=|X|(E_i)_{u_0x}$ ($x\in X$, $i=0,1,\dots,d$).
They form two bases of the \emph{dual Bose--Mesner algebra with respect to $u_0$}.
When we assume that $\mathfrak{X}$ is $P$-polynomial (resp.~$Q$-polynomial), we understand that $A_0,A_1,\dots,A_d$ (resp.~$E_0,E_1,\dots,E_d$) is the $P$-polynomial ordering (resp.~$Q$-polynomial ordering) and write $A=A_1=\sum_{i=0}^d\theta_iE_i$ (resp.~$A^*=A_1^*=\sum_{i=0}^d\theta_i^*E_i^*$).
The \emph{Terwilliger algebra} $T$ \emph{with respect to} $u_0$ is the subalgebra of the full matrix algebra generated by the Bose--Mesner algebra and the dual Bose--Mesner algebra.
We note that $T$ is semisimple since it is closed under conjugate-transposition.

The \emph{endpoint}, \emph{dual endpoint}, \emph{diameter}, and the \emph{dual diameter} of an irreducible $T$-module $W$ are defined by $\rho(W)=\min\{i\, |\,E_i^*W\ne 0\}$, $\rho^*(W)=\min\{i\, |\,E_iW\ne 0\}$, $\delta(W)=|\{i\, |\,E_i^*W\ne 0\}|-1$, and $\delta^*(W)=|\{i\, |\,E_iW\ne 0\}|-1$, respectively.\footnote{In \cite{Terwilliger1992JAC,Terwilliger1993JACa,Terwilliger1993JACb}, $\rho(W)$, $\rho^*(W)$, $\delta(W)$, and $\delta^*(W)$ are called the dual endpoint, endpoint, dual diameter, and the diameter of $W$, respectively.}
The module $W$ is called \emph{thin} (resp.~\emph{dual thin}) if $\dim(E_i^*W)\leq 1$ (resp.~$\dim(E_iW)\leq 1$) ($i=0,1,\dots,d$).
There is a unique irreducible $T$-module $W$ with $\rho(W)=0$ \emph{or} $\rho^*(W)=0$ up to isomorphism, that is to say, the \emph{primary} $T$-module $\operatorname{span}\{\hat{u}_0,A_1\hat{u}_0,\dots,A_d\hat{u}_0\}$; cf.~\cite[Lemma 3.6]{Terwilliger1992JAC}.
It is thin, dual thin, and has diameter and dual diameter both equal to $d$.
We call $\mathfrak{X}$ \emph{thin} (resp.~\emph{dual thin}) \emph{with respect to} $u_0$ if every irreducible $T$-module is thin (resp.~dual thin).\footnote{We simply call $\mathfrak{X}$ \emph{thin} (resp.~\emph{dual thin}) if it is thin (resp.~dual thin) with respect to every base vertex $u_0\in X$.}
The next two lemmas will be freely used in our discussions.

\begin{lem}[{\cite[Lemma 3.9]{Terwilliger1992JAC}}]
Suppose that $\mathfrak{X}$ is $P$-polynomial.
Let $W$ be an irreducible $T$-module and set $\rho=\rho(W)$, $\delta=\delta(W)$.
Then, the following hold:
\begin{enumerate}
\setlength{\itemsep}{0pt}
\item $AE_i^*W\subseteq E_{i-1}^*W+E_i^*W+E_{i+1}^*W$\, $(i=0,1,\dots,d)$, where $E_{-1}^*=E_{d+1}^*=0$.
\item $\{i\, |\,E_i^*W\ne 0\}=\{\rho,\rho+1,\dots,\rho+\delta\}$.
\item $E_i^*AE_j^*W\ne 0$ if $|i-j|=1$\, $(i,j=\rho,\dots,\rho+\delta)$.
\item If $W$ is thin, then $E_iW=E_iE_{\rho}^*W$\, $(i=0,1,\dots,d)$; in particular, $W$ is dual thin and $\delta^*(W)=\delta$.
\end{enumerate}
\end{lem}

\begin{lem}[{\cite[Lemma 3.12]{Terwilliger1992JAC}}]
Suppose that $\mathfrak{X}$ is $Q$-polynomial.
Let $W$ be an irreducible $T$-module and set $\rho^*=\rho^*(W)$, $\delta^*=\delta^*(W)$.
Then, the following hold:
\begin{enumerate}
\setlength{\itemsep}{0pt}
\item $A^*E_iW\subseteq E_{i-1}W+E_iW+E_{i+1}W$\, $(i=0,1,\dots,d)$, where $E_{-1}=E_{d+1}=0$.
\item $\{i\, |\,E_iW\ne 0\}=\{\rho^*,\rho^*+1,\dots,\rho^*+\delta^*\}$.
\item $E_iA^*E_jW\ne 0$ if $|i-j|=1$\, $(i,j=\rho^*,\dots,\rho^*+\delta^*)$.
\item If $W$ is dual thin, then $E_i^*W=E_i^*E_{\rho^*}W$\, $(i=0,1,\dots,d)$; in particular, $W$ is thin and $\delta(W)=\delta^*$.
\end{enumerate}
\end{lem}

We note that if $\mathfrak{X}$ is $P$-polynomial then
\begin{equation}\label{Hom in terms of T}
	\operatorname{Hom}_j(X)=\operatorname{span}\left\{\left.\!\left(\sum_{i=j}^dE_i^*A_{i-j}E_j^*\right)\hat{z}\ \right|\, z\in X_j\right\}=\left(\sum_{i=j}^dE_i^*A_{i-j}E_j^*\right)E_j^*\mathbb{C}^X
\end{equation}
for $j=0,1,\dots,d$.

\begin{thm}\label{criterion-P}
Suppose that $\mathfrak{X}$ is $P$-polynomial, and let $e, r_1, r_2, \dots, r_p$ be integers with $p-1\leq e\leq r_1< r_2<\dots< r_p\leq d$.
Suppose that every irreducible $T$-module $W$ with $\rho(W)\leq e$ is thin and satisfies $\rho(W)+\delta(W)\geq r_p$.
If the $p\times p$ matrix consisting of the intersection numbers $c_i=p_{1,i-1}^i$ $(i=1,2,\dots,d)$ defined by
\begin{equation}\label{p by p}
	\begin{pmatrix} 1& c_{r_1-e+p-1} & \dots & (c_{r_1-e+p-1}\dots c_{r_1-e+1}) \\ \vdots & \vdots && \vdots \\ 1 & c_{r_p-e+p-1} & \dots & (c_{r_p-e+p-1}\dots c_{r_p-e+1}) \end{pmatrix}
\end{equation}
(where the $(i,j)$-entry is $\prod_{h=1}^{j-1}c_{r_i-e+p-h}$) is nonsingular, then
\begin{equation*}
	\dim\big(\!\operatorname{Hom}_{0}(S)+\operatorname{Hom}_{1}(S)+\dots+\operatorname{Hom}_{e}(S)\big)=k_e+k_{e-1}+\dots + k_{e-p+1},
\end{equation*}
where $S=X_{r_1}\cup X_{r_2}\cup\dots\cup X_{r_p}$.
\end{thm}

\begin{proof}
Fix a set $\mathcal{W}$ of irreducible $T$-modules in $\mathbb{C}^X$ such that $\mathbb{C}^X=\bigoplus_{W\in\mathcal{W}}W$.
Observe that
\begin{equation*}
	E_j^*\mathbb{C}^X=\bigoplus_{\substack{W\in\mathcal{W} \\ \rho(W)\leq j}} E_j^*W \quad (j=0,1,\dots,d),
\end{equation*}
so that by \eqref{Hom in terms of T} we have
\begin{align*}
	\operatorname{Hom}_j(S) & = \left(\sum_{i=1}^p E_{r_i}^*A_{r_i-j}E_j^*\right) E_j^*\mathbb{C}^X  \\
	&=\bigoplus_{\substack{W\in\mathcal{W} \\ \rho(W)\leq j}} \left(\sum_{i=1}^p E_{r_i}^*A_{r_i-j}E_j^*\right)E_j^*W \quad (j=0,1,\dots,e).
\end{align*}
In particular, it follows that
\begin{equation*}
	\sum_{j=0}^e\operatorname{Hom}_j(S) = \bigoplus_{\substack{ W\in\mathcal{W} \\ \rho(W)\leq e }} \left(\sum_{j=0}^e\operatorname{Hom}_j(S)\right)\cap W,
\end{equation*}
and that
\begin{equation*}
	\left(\sum_{j=0}^e\operatorname{Hom}_j(S)\right)\cap W = \sum_{j=\rho(W)}^e \left(\sum_{i=1}^p E_{r_i}^*A_{r_i-j}E_j^*\right)E_j^*W \subseteq \sum_{i=1}^p E_{r_i}^*W
\end{equation*}
for every $W\in\mathcal{W}$ with $\rho(W)\leq e$.

Pick any $W\in\mathcal{W}$ with $\rho:=\rho(W)\leq e$, and let $v$ be a nonzero vector in $E_{\rho}^*W$.
Recall that $\{i\,|\,E_i^*W\ne 0\}=\{\rho,\dots,\rho+\delta\}$, where $\delta=\delta(W)$.
First, suppose that $\rho\leq e-p+1$.
Since $W$ is thin and since $\rho+\delta\geq e$, for $j=e-p+1,\dots,e$, the vector $v_j=E_j^*A^{j-\rho}v$ is nonzero and hence is a basis of $E_j^*W$.
Moreover, for $j=e-p+1,\dots,e$, it follows that
\begin{align*}
	\left(\sum_{i=1}^p E_{r_i}^*A_{r_i-j}E_j^*\right) v_j &= \sum_{i=1}^p \frac{1}{c_{r_i-j}\dots c_2c_1} E_{r_i}^*A^{r_i-j}E_j^* v_j \\
	&= \sum_{i=1}^p \frac{1}{c_{r_i-j}\dots c_2c_1} E_{r_i}^*A^{r_i-\rho}v \\
	&= \sum_{i=1}^p \frac{1}{c_{r_i-j}\dots c_2c_1} E_{r_i}^*A^{r_i-e+p-1}E_{e-p+1}^* A^{e-p+1-\rho}v \\
	&= \sum_{i=1}^p (c_{r_i-e+p-1}\dots c_{r_i-j+1}) E_{r_i}^*A_{r_i-e+p-1}v_{e-p+1},
\end{align*}
where we have used the fact that $c_n\dots c_2c_1$ is the number of the geodesics between two vertices at distance $n$ (in the distance-regular graph $(X,R_1)$).
Since $r_1,r_2,\dots,r_p\in \{\rho,\dots,\rho+\delta\}$, the vectors $E_{r_i}^*A_{r_i-e+p-1}v_{e-p+1}$ $(i=1,2,\dots,p)$ are nonzero and hence form a basis of $\sum_{i=1}^p E_{r_i}^*W$.
Thus, since the coefficient matrix \eqref{p by p} is nonsingular, the vectors $\left(\sum_{i=1}^p E_{r_i}^*A_{r_i-j}E_j^*\right) v_j$ $(j=e-p+1,\dots,e)$ also form a basis of $\sum_{i=1}^p E_{r_i}^*W$.
It follows that $\big(\sum_{j=0}^e\operatorname{Hom}_j(S)\big)\cap W=\sum_{i=1}^p E_{r_i}^*W$.
In particular, $\dim\big(\big(\sum_{j=0}^e\operatorname{Hom}_j(S)\big)\cap W\big)=p$.
Next, suppose that $e-p+2\leq \rho\leq e$.
Likewise, using the fact that the last $(e-\rho+1)$ columns of the matrix \eqref{p by p} are linearly independent, we find that the vectors $\big(\sum_{i=1}^p E_{r_i}^*A_{r_i-j}E_j^*\big) v_j$ $(j=\rho,\dots,e)$ are linearly independent, and hence that $\dim\big(\big(\sum_{j=0}^e\operatorname{Hom}_j(S)\big)\cap W\big)=e-\rho+1$.
Thus, it follows that
\begin{align*}
	\dim\left(\sum_{j=0}^e\operatorname{Hom}_j(S)\right) &= \sum_{\substack{ W\in\mathcal{W} \\ \rho(W)\leq e }}\min\{p,e-\rho(W)+1\} \\
	&= \sum_{\substack{ W\in\mathcal{W} \\ \rho(W)\leq e }} \sum_{j=e-p+1}^e \dim (E_j^*W) \\
	&= \sum_{j=e-p+1}^e \dim (E_j^*\mathbb{C}^X) \\
	&= \sum_{j=e-p+1}^e k_j,
\end{align*}
as desired.
\end{proof}

\begin{rem}
In view of \cite[Lemma 5.1]{Caughman1999DM}, the assumption $\rho(W)+\delta(W)\geq r_p$ in Theorem \ref{criterion-P} holds provided that $r_p\leq d-e$.
\end{rem}

\begin{ex}\label{Hamming-P}
Suppose that $\mathfrak{X}$ is a Hamming scheme $H(d,q)$.
Then, $c_i=i$ $(i=1,2,\dots,d)$.
Thus, it follows that the matrix \eqref{p by p} is essentially Vandermonde (in the variables $r_1,r_2,\dots,r_p$), and hence is nonsingular.
We note that $\mathfrak{X}$ is thin; cf.~\cite[Example 6.1]{Terwilliger1993JACb}.
\end{ex}

\begin{ex}
Suppose that $\mathfrak{X}$ is a dual polar scheme.
Then, $c_i=(q^i-1)/(q-1)$ $(i=1,2,\dots,d)$ for some prime power $q\geq 2$.
Thus, the matrix \eqref{p by p} is again essentially Vandermonde (in the variables $q^{r_1},q^{r_2},\dots,q^{r_p}$), and hence is nonsingular.
We note that $\mathfrak{X}$ is thin; cf.~\cite[Example 6.1]{Terwilliger1993JACb}.
In Appendix \ref{sec: dual-polar}, we show that Theorem \ref{Fisher-P} is valid for dual polar schemes.
\end{ex}

Next, we move on to the $Q$-polynomial case.

\begin{thm}\label{criterion-Q}
Suppose that $\mathfrak{X}$ is $Q$-polynomial, and let $e, r_1, r_2, \dots, r_p$ be integers with $p-1\leq e\leq d$ and $0\leq r_1< r_2<\dots< r_p\leq d$.
If every irreducible $T$-module $W$ with $\rho^*(W)\leq e$ is dual thin, and satisfies $\rho^*(W)+\delta^*(W)\geq e$ and $|\{i\, |\, E_{r_i}^*W\ne 0\}|\geq \min\{p,e-\rho^*(W)+1\}$, then
\begin{equation*}
	\dim\big(L_0 (S)+L_1 (S)+\dots+L_e (S)\big)=m_e+m_{e-1}+\dots + m_{e-p+1},
\end{equation*}
where $S=X_{r_1}\cup X_{r_2}\cup\dots\cup X_{r_p}$.
\end{thm}

\begin{proof}
Again, fix a set $\mathcal{W}$ of irreducible $T$-modules in $\mathbb{C}^X$ such that $\mathbb{C}^X=\bigoplus_{W\in\mathcal{W}}W$.
Observe that
\begin{equation*}
	L_j(X)=E_j\mathbb{C}^X=\bigoplus_{\substack{W\in\mathcal{W} \\ \rho^*(W)\leq j}} E_jW \quad (j=0,1,\dots,d),
\end{equation*}
so that
\begin{equation*}
	L_j(S)=\left(\sum_{i=1}^p E_{r_i}^* \right) E_j\mathbb{C}^X =\bigoplus_{\substack{W\in\mathcal{W} \\ \rho^*(W)\leq j}} \left(\sum_{i=1}^p E_{r_i}^* \right)E_jW \quad (j=0,1,\dots,d).
\end{equation*}
In particular, it follows that
\begin{equation*}
	\sum_{j=0}^eL_j(S) = \bigoplus_{\substack{ W\in\mathcal{W} \\ \rho^*(W)\leq e }} \left(\sum_{j=0}^eL_j(S)\right)\cap W,
\end{equation*}
and that
\begin{equation*}
	\left(\sum_{j=0}^eL_j(S)\right)\cap W = \left(\sum_{i=1}^p E_{r_i}^* \right) \sum_{j=\rho^*(W)}^e E_jW \subseteq \sum_{i=1}^p E_{r_i}^*W
\end{equation*}
for every $W\in\mathcal{W}$ with $\rho^*(W)\leq e$.

Pick any $W\in\mathcal{W}$ with $\rho^*:=\rho^*(W)\leq e$, and let $v$ be a nonzero vector in $E_{\rho^*}W$.
First, suppose that $\rho^*\leq e-p+1$.
Then, $v,A^*v,\dots,A^{*p-1}v\in \sum_{j=\rho^*}^eE_jW$.
Since $W$ is dual thin, $\{E_i^*v\, |\, E_i^*W\ne 0\}$ is an orthogonal basis of $W$.
We note that $E_{r_i}^*W\ne 0$ for $i=1,2,\dots,p$.
Thus, the vectors $\big(\sum_{i=1}^p E_{r_i}^* \big)A^{*h}v=\sum_{i=1}^p \theta_{r_i}^{*h} E_{r_i}^*v$ $(h=0,1,\dots,p-1)$ belong to $\big(\sum_{j=0}^eL_j(S)\big)\cap W$ and form a basis of $\sum_{i=1}^p E_{r_i}^*W$, since the coefficient matrix is Vandermonde.
It follows that $\big(\sum_{j=0}^eL_j(S)\big)\cap W=\sum_{i=1}^p E_{r_i}^*W$.
In particular, $\dim\big(\big(\sum_{j=0}^eL_j(S)\big)\cap W\big)=p$.
Next, suppose that $e-p+2\leq \rho^*\leq e$.
Likewise, we find that the vectors $\big(\sum_{i=1}^p E_{r_i}^* \big)A^{*h}v$ $(h=0,1,\dots,e-\rho^*)$ belong to $\big(\sum_{j=0}^eL_j(S)\big)\cap W$ and are linearly independent, from which it follows that $\dim\big(\big(\sum_{j=0}^eL_j(S)\big)\cap W\big)=e-\rho^*+1$.
Thus, it follows that
\begin{align*}
	\dim\left(\sum_{j=0}^eL_j(S)\right) &= \sum_{\substack{ W\in\mathcal{W} \\ \rho^*(W)\leq e }}\min\{p,e-\rho^*(W)+1\} \\
	&= \sum_{\substack{ W\in\mathcal{W} \\ \rho^*(W)\leq e }} \sum_{j=e-p+1}^e \dim (E_jW) \\
	&= \sum_{j=e-p+1}^e \dim (E_j\mathbb{C}^X) \\
	&= \sum_{j=e-p+1}^e m_j,
\end{align*}
where the second equality follows since every $W\in\mathcal{W}$ with $\rho^*(W)\leq e$ is dual thin and satisfies $\{\rho^*(W),\dots,e\}\subseteq\{\rho^*(W),\dots,\rho^*(W)+\delta^*(W)\}=\{j\,|\,E_jW\ne 0\}$.
This completes the proof.
\end{proof}

\begin{rem}
In view of \cite[Lemma 7.1]{Caughman1999DM}, the assumption $\rho^*(W)+\delta^*(W)\geq e$ in Theorem \ref{criterion-Q} holds provided that $e\leq \lceil d/2\rceil$.
\end{rem}

\begin{ex}
Suppose that $\mathfrak{X}$ is a Hamming scheme $H(d,q)$.
Then, $\rho(W)=\rho^*(W)$ for every irreducible $T$-module $W$; cf.~\cite[Example 6.1]{Terwilliger1993JACb}.
Thus, the assumption of Theorem \ref{criterion-Q} is satisfied provided that $e\leq r_1<r_2<\dots<r_p\leq d-e$.
Of course, in this case the conclusion also follows from Proposition \ref{Hom-L}, Theorem \ref{criterion-P}, and Example \ref{Hamming-P}.
\end{ex}

\begin{ex}
Suppose that $\mathfrak{X}$ is $P$-polynomial, $Q$-polynomial, and bipartite.
In this case, Caughman \cite{Caughman1999DM} showed that $\mathfrak{X}$ is thin, dual thin, and that every irreducible $T$-module $W$ satisfies $\delta(W)=\delta^*(W)=d-2\rho^*(W)$ and $\rho^*(W)\leq \rho(W)\leq 2\rho^*(W)$.
Thus, the assumption of Theorem \ref{criterion-Q} is satisfied provided that $2e\leq r_1<r_2<\dots<r_p\leq d-e$.
\end{ex}

\begin{ex}
Suppose that $\mathfrak{X}$ is a Johnson scheme $J(v,d)$.
Then, $\mathfrak{X}$ is thin, dual thin, and every irreducible $T$-module $W$ satisfies $\rho(W)\leq \rho^*(W)$; cf.~\cite[Example 6.1]{Terwilliger1993JACb}.
Thus, in view of \cite[Lemma 5.1]{Caughman1999DM}, the assumption of Theorem \ref{criterion-Q} is satisfied provided that $e\leq r_1<r_2<\dots<r_p\leq d-e$.
\end{ex}

We note that if some of the assumptions on the irreducible $T$-modules in Theorems \ref{criterion-P} and \ref{criterion-Q} are not satisfied, then the dimensions of the subspaces in question can indeed be smaller.
For example, we have the following result:

\begin{prop}\label{e=p=1}
Suppose that $\mathfrak{X}$ is $P$-polynomial and $Q$-polynomial, and let $S=X_d$.
Let $\eta_1\geq \eta_2\geq\dots\geq \eta_{k_1}$ be the eigenvalues\footnote{The $\eta_i$ are the eigenvalues of the subgraph of $(X,R_1)$ induced on $X_1$ (called the \emph{local graph}), which is regular with valency $\eta_1=a_1=p_{11}^1$.} of $E_1^*AE_1^*$ on $E_1^*\mathbb{C}^X$.
For every $\theta\in\mathbb{C}\cup\{\infty\}$, let $\tilde{\theta}=-1-b_1/(1+\theta)$ (a M\"{o}bius transformation) where $b_1=p_{12}^1$, and define $\mu_{\theta}=|\{i\geq 2\,|\,\eta_i=\tilde{\theta}\}|$.
Then, $\dim\big(\!\operatorname{Hom}_0(S)+\operatorname{Hom}_1(S)\big)=k_1-\mu_{\theta_1}-\mu_{\theta_d}$ and $\dim\big(L_0(S)+L_1(S)\big)=m_1-\mu_{\theta_d}$.
\end{prop}

\begin{proof}
Let $\mathcal{W}$ be as in the proofs of Theorems \ref{criterion-P} and \ref{criterion-Q}.
Recall that $A$ and $A^*$ act on every $W\in\mathcal{W}$ as a tridiagonal pair in the sense of \cite{ITT2001P}; cf.~\cite[Example 1.4]{ITT2001P}.
In particular, by \cite[Lemma 4.5]{ITT2001P} we have $\delta(W)=\delta^*(W)$ for $W\in\mathcal{W}$.

Let $\mathcal{W}_1=\{W\in\mathcal{W}\,|\,\rho(W)=1,\rho^*(W)=2,\delta(W)=d-2\}$, $\mathcal{W}_d=\{W\in\mathcal{W}\,|\,\rho(W)=1,\rho^*(W)=1,\delta(W)=d-2\}$.
Let $\theta_{\mathrm{sec}},\theta_{\mathrm{min}}$ be the second largest and the smallest eigenvalues of $A$, respectively.
Then, in view of \cite[Lemma 8.5]{GT2002EJC}, it follows that the condition that $\mathcal{W}_1\ne\emptyset$ (resp.~$\mathcal{W}_d\ne\emptyset$) implies that $\theta_1\in\{\theta_{\mathrm{sec}},\theta_{\mathrm{min}}\}$ (resp.~$\theta_d\in\{\theta_{\mathrm{sec}},\theta_{\mathrm{min}}\}$).
Next, observe that $\tilde{\theta}_i\not\in (\tilde{\theta}_{\mathrm{sec}},\tilde{\theta}_{\mathrm{min}})$ $(i=1,2,\dots,d)$.
On the other hand, by \cite[Theorem 8.4]{GT2002EJC} we have $\tilde{\theta}_{\mathrm{sec}}\leq\eta_i\leq\tilde{\theta}_{\mathrm{min}}$ $(i=2,3,\dots,k_1)$.
Thus, the condition that $\mu_{\theta_1}>0$ (resp.~$\mu_{\theta_d}>0$) implies again that $\theta_1\in\{\theta_{\mathrm{sec}},\theta_{\mathrm{min}}\}$ (resp.~$\theta_d\in\{\theta_{\mathrm{sec}},\theta_{\mathrm{min}}\}$).
With these explained, it follows from \cite[Lemma 8.5, Theorems 9.8, 10.1, 11.5]{GT2002EJC} that every $W\in\mathcal{W}_1\cup\mathcal{W}_d$ is thin, and that $\mu_{\theta_1}=|\mathcal{W}_1|$ and $\mu_{\theta_d}=|\mathcal{W}_d|$.

Let $W\in\mathcal{W}$ and write $\rho=\rho(W)$, $\rho^*=\rho^*(W)$, and $\delta=\delta(W)=\delta^*(W)$.
Recall that $\{i\,|\,E_i^*W\ne 0\}=\{\rho,\dots,\rho+\delta\}$ and $\{i\,|\,E_iW\ne 0\}=\{\rho^*,\dots,\rho^*+\delta\}$.
By \cite[Theorem 1.3]{NT2008LAAd}, we have $\dim E_{\rho}^*W=\dim E_{\rho^*}W=\dim E_{\rho+\delta}^*W=\dim E_{\rho^*+\delta}W=1$.

We first compute $\dim\big(\!\operatorname{Hom}_0(S)+\operatorname{Hom}_1(S)\big)$.
Suppose that $\rho=0$.
Then, $\delta=d$ and $W$ is the (thin) primary $T$-module.
It follows that $\big(E_d^*A_dE_0^*\big)E_0^*W+\big(E_d^*A_{d-1}E_1^*\big)E_1^*W=E_d^*W\ne 0$.
Next, suppose that $\rho=1$.
By \cite[Lemma 5.1]{Caughman1999DM}, we have $\delta\in\{d-2,d-1\}$.
Observe that $\delta=d-2$ precisely when $W\in\mathcal{W}_1\cup\mathcal{W}_d$.
Terwilliger \cite[Lecture 34]{Terwilliger1993N} showed that $E_d^*W=\big(E_d^*A_{d-1}E_1^*\big)E_1^*W$, from which it follows that $\big(E_d^*A_{d-1}E_1^*\big)E_1^*W=0$ if and only if $\delta=d-2$, i.e., $W\in\mathcal{W}_1\cup\mathcal{W}_d$.
Thus, as in the proof of Theorem \ref{criterion-P}, it follows that $\dim\big(\!\operatorname{Hom}_0(S)+\operatorname{Hom}_1(S)\big)=k_1-|\mathcal{W}_1\cup\mathcal{W}_d|=k_1-\mu_{\theta_1}-\mu_{\theta_d}$.

We now compute $\dim\big(L_0(S)+L_1(S)\big)$.
We note that $E_d^*W=E_d^*E_{\rho^*}W$ in view of \cite[Lemma 5.1]{NT2008LAAb}.
Suppose that $\rho^*=0$.
Then, $\delta=d$ and $W$ is again the primary $T$-module.
It follows that $E_d^*(E_0W+E_1W)=E_d^*W\ne 0$.
Next, suppose that $\rho^*=1$.
By \cite[Lemma 7.1]{Caughman1999DM}, we have $\delta\in\{d-2,d-1\}$.
It follows that $E_d^*E_1W=0$ if and only if $W\in\mathcal{W}_d$.
Thus, as in the proof of Theorem \ref{criterion-Q}, it follows that $\dim\big(L_0(S)+L_1(S)\big)=m_1-|\mathcal{W}_d|=m_1-\mu_{\theta_d}$.
\end{proof}

\begin{ex}
Suppose that $\mathfrak{X}$ is a Hamming scheme $H(d,q)$.
Then, $k_1=m_1=d(q-1)$, $b_1=(d-1)(q-1)$, $\theta_i=q(d-i)-d$ ($i=0,1,\dots,d$), and it is easy to see that $\mu_{\theta_1}=0$ and $\mu_{\theta_d}=d-1$.
We note that relative $2$-designs supported by $X_d$ (in the sense of both Definition \ref{design-P} and Definition \ref{design-Q}) are precisely the $2$-designs (i.e., orthogonal arrays with strength $2$) in the Hamming scheme $H(d,q-1)$ induced on $X_d$, and Proposition \ref{e=p=1} gives the Rao bound $1+d(q-2)$.
\end{ex}

\begin{ex}
Suppose that $\mathfrak{X}$ is a Johnson scheme $J(v,d)$.
Then, $k_1=d(v-d)$, $m_1=v-1$, $b_1=(d-1)(v-d-1)$, $\theta_i=(d-i)(v-d-i)-i$ ($i=0,1,\dots,d$), and it is easy to see that $\mu_{\theta_1}=(d-1)(v-d-1)$ and $\mu_{\theta_d}=d-1$.
We note that relative $2$-designs supported by $X_d$ (in the sense of both Definition \ref{design-P} and Definition \ref{design-Q}) are precisely the $2$-designs in the Johnson scheme $J(v-d,d)$ induced on $X_d$, and Proposition \ref{e=p=1} gives the Fisher bound $v-d$.
\end{ex}

\section{A characterization of Hamming schemes}\label{sec: Hamming}

In this section, for $d\geq 6$, we characterize the Hamming schemes $H(d,q)$ as the formally self-dual $P$- and $Q$-polynomial association schemes with the property that $\operatorname{Hom}_0 (X)+\operatorname{Hom}_1(X)=L_0 (X)+L_1(X)$.
We begin with the following result:

\begin{prop}\label{criterion-Hom-L}
Suppose that $\mathfrak{X}$ is $P$-polynomial, $Q$-polynomial, and that $\operatorname{Hom}_0 (X)+\operatorname{Hom}_1(X)=L_0 (X)+L_1(X)$.
Then, $c_i/(\theta_i^*-\theta_0^*)$ is independent of $i=1,2,\dots,d$.
\end{prop}

\begin{proof}
By \eqref{Hom in terms of T} and since $A\hat{u}_0\in E_1^*\mathbb{C}^X$, the vector
\begin{equation*}
	\left(\sum_{i=1}^d E_i^*A_{i-1} E_1^*\right) A\hat{u}_0 = \sum_{i=1}^d E_i^* A_{i-1}A\hat{u}_0 = \sum_{i=1}^d c_i A_i \hat{u}_0
\end{equation*}
belongs to $\operatorname{Hom}_1(X)\subset L_0(X)+L_1(X)$.
On the other hand, this vector is in the primary $T$-module $\operatorname{span}\{\hat{u}_0,A_1\hat{u}_0,\dots,A_d\hat{u}_0\}=\operatorname{span}\{E_0\hat{u}_0,E_1\hat{u}_0,\dots,E_d\hat{u}_0\}$.
Thus, it is written as
\begin{equation*}
	\sum_{i=1}^d c_i A_i \hat{u}_0 = \alpha E_0\hat{u}_0 + \beta E_1\hat{u}_0 = \frac{1}{|X|} \sum_{i=0}^d (\alpha+\beta\theta_i^*)A_i\hat{u}_0,
\end{equation*}
for some $\alpha,\beta\in\mathbb{C}$.
Comparing the coefficients of $\hat{u}_0$, we find $\beta=-\alpha/\theta_0^*$, and hence
\begin{equation*}
	c_i=\frac{\alpha}{|X|\theta_0^*}(\theta_0^*-\theta_i^*) \quad (i=1,2,\dots,d),
\end{equation*}
as desired.\footnote{In fact, we have $\alpha=\sum_{i=1}^dc_ik_i$.}
\end{proof}

Using this result, we now prove the following theorem:

\begin{thm}\label{Hamming}
Suppose that $\mathfrak{X}$ is formally self-dual, $P$-polynomial (and $Q$-polynomial), and satisfies $\operatorname{Hom}_0 (X)+\operatorname{Hom}_1(X)=L_0 (X)+L_1(X)$.
If $d\geq 6$, then $\mathfrak{X}$ is the Hamming scheme $H(d,q)$ for some $q$.
\end{thm}

\begin{proof}
Since $\mathfrak{X}$ is formally self-dual, in the notation of \cite[Section 3.5]{BI1984B} and \cite[Section 2]{Terwilliger1992JAC}, the parameters of $\mathfrak{X}$ satisfy one of the following cases\footnote{In the terminology of \cite{Terwilliger2005DCC}, these are of $q$-Racah, affine $q$-Krawtchouk, Racah, Krawtchouk and Bannai/Ito types, respectively.}: (I) with $s=s^*\ne 0$; (I) with $s=s^*=0$; (II) with $s=s^*$; (IIC); and (III) with $s=s^*$.

First, consider Case (I) with $s=s^*\ne 0$.
Then, it follows that
\begin{equation*}
	\frac{c_i}{\theta_i^*-\theta_0^*}=\frac{q^i(1-sq^{i+d+1})(r_1-sq^i)(r_2-sq^i)}{sq^d(1-sq^{i+1})(1-sq^{2i})(1-sq^{2i+1})} \quad (i=1,2,\dots,d-1),
\end{equation*}
and this is independent of $i$ by Proposition \ref{criterion-Hom-L}, so that
\begin{equation*}
	sq^d(1-sq^{i+1})(1-sq^{2i})(1-sq^{2i+1}) = (\theta_1^*-\theta_0^*)q^i(1-sq^{i+d+1})(r_1-sq^i)(r_2-sq^i)
\end{equation*}
for $i=1,2,\dots,d-1$, and this identity is valid for $i=d$ as well.
However, as polynomials in $q^i$, the left hand side is of degree five, whereas the right hand side is of degree four.
Since $d\geq 6$, this is impossible.
Case (I) with $s=s^*=0$ is ruled out in the same way.

Next, consider Case (II) with $s=s^*$.
Then, it follows that
\begin{equation*}
	\frac{c_i}{\theta_i^*-\theta_0^*}=\frac{(i+s+d+1)(i+s-r_1)(i+s-r_2)}{(i+1+s)(2i+1+s)(2i+s)} \quad (i=1,2,\dots,d-1).
\end{equation*}
Again, as polynomials in $i$, the denominator must be a scalar multiple of the numerator.
In particular, they have the same roots.
Since $1+s\ne s+d+1$, we may assume that $1+s=s-r_1$, i.e., $r_1=-1$.
Then, since $r_1+r_2=s+s^*+d+1$, we have $r_2=2s+d+2$.
Using this and $\{s+d+1,s-r_2\}=\{(1+s)/2,s/2\}$, it follows that $d=\pm 1/4$, which is absurd.

If $\mathfrak{X}$ satisfies Case (III) with $s=s^*$, then by the classification due to Terwilliger \cite{Terwilliger1987JCTB}, it follows that $\mathfrak{X}$ is isomorphic to $H(d,2)$ ($d$ even) or the bipartite half of $H(2d+1,2)$, but with respect to the second $P$-polynomial orderings.\footnote{The second $P$-polynomial ordering of the Johnson scheme $J(2d+1,d)$ (corresponding to the Odd graph $O_{d+1}$) satisfies Case (III), but with $s=2d+3$ and $s^*=2d+2$.}
We have $c_i=i$ $(i=1,2,\dots,d)$ in either case, and it follows that $c_i/(\theta_i^*-\theta_0^*)$ cannot be constant, since $\theta_0^*,\theta_1^*,\dots,\theta_d^*$ are not an arithmetic progression.

Thus, we are left with Case (IIC).
In this case, by the classification due to Egawa \cite{Egawa1981JCTA}, $\mathfrak{X}$ is a Hamming scheme or a Doob scheme.
If $\mathfrak{X}$ is a Hamming scheme, then we are done.
Thus, suppose that $\mathfrak{X}$ is a Doob scheme.
Then, there is a thin irreducible $T$-module $W$ with $\rho(W)=1$, $\rho^*(W)=2$, and $\delta(W)=d-2$.
This fact follows from Tanabe's description \cite{Tanabe1997JAC} of the irreducible $T$-modules of Doob schemes, but we may also prove it as follows.
The local graph of the Doob graph $(X,R_1)$ (whose adjacency matrix is essentially $E_1^*AE_1^*$) is a disjoint union of hexagons and $3$-cliques, so that it has $-2$ as an eigenvalue.
On the other hand, we have $-1-b_1/(1+\theta_1)=-2$, where $b_1=p_{12}^1$.
Thus, by \cite[Theorem 9.8]{GT2002EJC}, any eigenvector (in $E_1^*\mathbb{C}^X$) of $E_1^*AE_1^*$ with eigenvalue $-2$ generates such a $T$-module.
Now, let $v$ be a nonzero vector in $E_1^*W$.
Then, $\big(\sum_{i=1}^dE_i^*A_{i-1} E_1^*\big)v$ is nonzero and belongs to $\operatorname{Hom}_1(X)$.
However, since $\rho^*(W)=2$, it is contained in $L_2(X)+L_3(X)+\dots+L_d(X)$.
Thus, we conclude that $\operatorname{Hom}_0(X)+\operatorname{Hom}_1(X)\ne L_0(X)+L_1(X)$, and the proof is complete.
\end{proof}

\appendix

\section{Comments on Theorem \ref{Fisher-P}} \label{sec: dual-polar}

In this appendix, we generalize Theorem \ref{Fisher-P} to dual polar schemes (Theorem \ref{a6}).
Suppose that $\mathfrak{X}$ is a dual polar scheme, so that $X$ is the set of maximal isotropic subspaces of a vector space $V$ over a finite field, equipped with a non-degenerate form (alternating, Hermitian, or quadratic) of Witt index $d$.
For convenience, we shall work with the dual polar graph $(X,R_1)$ with path-length distance $\partial$.

\begin{lem}\label{a1}
Let $x,y,z\in X$.
Then, $\partial(x,z)+\partial(z,y)=\partial(x,y)$ if and only if $x\cap y\subseteq z=(x\cap z)+(y\cap z)$.
\end{lem}

\begin{proof}
Immediate from $\dim (x\cap z)+\dim (y\cap z)\leq d+\dim (x\cap y\cap z)\leq d+\dim (x\cap y)$.
\end{proof}

For the moment, fix $x,y\in X$ and write $i=\partial(u_0,x)$, $j=\partial(u_0,y)$, $h=\partial(x,y)$, and $\ell=\dim (u_0\cap U)$, where $U=x\cap y$.
We note that $\ell\geq d-i-j$.
Our goal is to show that $f_xf_y\in\operatorname{Hom}_{d-\ell}(X)$.
We set $X'=\{z\in X\, |\, U\subseteq z \}$, and observe that $X'$ induces a dual polar graph with diameter $h$.

\begin{lem}\label{a2}
For every $z\in X$, there is a unique $z'\in X'$ such that $\partial(z,z')=\partial(z,X')$.
Moreover, it holds that $\partial(z,z_1)=\partial(z,z')+\partial(z',z_1)$ for all $z_1\in X'$.
\end{lem}

\begin{proof}
Set $z'=U+(z\cap U^{\perp})\in X'$.
Pick any $z_1\in X'$.
Then, $z'=(z\cap z')+(z_1\cap z')$ since $U\subseteq z_1$ and $z\cap U^{\perp}\subseteq z$.
Moreover, $z\cap z_1\subseteq z\cap U^{\perp}\subseteq z'$.
Thus, $\partial(z,z_1)=\partial(z,z')+\partial(z',z_1)$ by Lemma \ref{a1}, and the result follows.
\end{proof}

\begin{lem}\label{a3}
Suppose that $z\in X'$ satisfies $f_x(z)=f_y(z)=1$.
Then, $\partial(u_0,z)=d-\ell$.
\end{lem}

\begin{proof}
First, $u_0\cap z\subseteq x$ by Lemma \ref{a1} and since $\partial(u_0,x)+\partial(x,z)=\partial(u_0,z)$.
Likewise, $u_0\cap z\subseteq y$.
Thus, $u_0\cap z\subseteq u_0\cap U$.
On the other hand, $u_0\cap U\subseteq u_0\cap z$ since $z\in X'$.
It follows that $u_0\cap z=u_0\cap U$, as desired.
\end{proof}

\begin{lem}\label{a4}
For every $z\in X$ such that $f_x(z)=f_y(z)=1$, there is a unique $z'\in X'$ such that $f_x(z')=f_y(z')=1$ and $f_{z'}(z)=1$.
\end{lem}

\begin{proof}
Let $z'(=U+(z\cap U^{\perp}))$ be as in Lemma \ref{a2}.
Then, $f_x(z')=f_y(z')=1$ and $f_{z'}(z)=1$.
To show the uniqueness, suppose that $z_1\in X'$ satisfies $f_x(z_1)=f_y(z_1)=1$ and $f_{z_1}(z)=1$.
Then, it follows from Lemma \ref{a3} that $\partial(u_0,z')=\partial(u_0,z_1)=d-\ell$, so that $\partial(z,z')=\partial(z,z_1)$.
But then, we must have $z'=z_1$ by Lemma \ref{a2}, and the proof is complete.
\end{proof}

\begin{prop}\label{a5}
With the above notation, it holds that
\begin{equation*}
	f_xf_y=\sum_{\substack{z\in X' \\ f_x(z)=f_y(z)=1}} f_z \in\operatorname{Hom}_{d-\ell}(X).
\end{equation*}
\end{prop}

\begin{proof}
Immediate from Lemmas \ref{a3} and \ref{a4}.
\end{proof}

\begin{thm}\label{a6}
Theorem \ref{Fisher-P} is valid for dual polar schemes.
\end{thm}

\begin{proof}
Suppose that $f\in \operatorname{Hom}_0(X)+\operatorname{Hom}_1(X)+\dots+\operatorname{Hom}_e(X)$ satisfies $f|_Y\equiv 0$.
Then, $f^2\in \operatorname{Hom}_0(X)+\operatorname{Hom}_1(X)+\dots+\operatorname{Hom}_{2e}(X)$ by Proposition \ref{a5}.
Thus,
\begin{equation*}
	\sum_{i=1}^p\frac{w(Y_{r_i})}{k_{r_i}}\sum_{x\in X_{r_i}}(f(x))^2=\sum_{y\in Y}w(y)(f(y))^2=0,
\end{equation*}
from which it follows that the restriction map $\operatorname{Hom}_0(S)+\operatorname{Hom}_1(S)+\dots+\operatorname{Hom}_e(S)\rightarrow \mathcal{F}(Y)$ ($f|_S\mapsto f|_Y$) is injective, and the result follows by comparing the dimensions.
\end{proof}

\section{Comments on Proposition \ref{Hom-L}} \label{sec: Hom-L}

We use the notation in the proof of Proposition \ref{Hom-L}.
We mentioned there that the function $f_z$ belongs to $\sum_{i=0}^jL_i(X)$.
While this fact is just a special case of a more general result about regular semilattices \cite{Delsarte1976JCTA,Delsarte1977PRR,Stanton1985JCTA}, we now provide an independent proof.

We identify $\{0,1,\dots,q-1\}$ with the additive group $\mathbb{Z}/q\mathbb{Z}$.
Let $\zeta\in\mathbb{C}$ be a primitive $q^{\mathrm{th}}$ root of unity.
Then, the \emph{additive group} $X$ and its dual group $X^*$ are isomorphic, and an isomorphism is given by $x=(x_1,x_2,\dots,x_d)\mapsto \varepsilon_x$, where $\varepsilon_x(y)=\zeta^{\sum_{\ell=1}^dx_{\ell}y_{\ell}}$ for every $y=(y_1,y_2,\dots,y_d)\in X$.
In fact, it is well known (and is easily checked) that $L_i(X)=\operatorname{span}\{\varepsilon_x\, |\, x\in X_i\}$ (over $\mathbb{C}$) for $i=0,1,\dots,d$, i.e., $H(d,q)$ is self-dual.

Assume that $i>j$, and pick any $y=(y_1,y_2,\dots,y_d)\in X_i$.
Then, the (standard) Hermitian inner product between $\varepsilon_y$ and $f_z$ is given by
\begin{equation*}
	\left(\prod_{h=1}^j\zeta^{z_{\ell_h}y_{\ell_h}}\right) \left(\prod_{\ell\ne \ell_1,\dots,\ell_j}\left(\sum_{x_{\ell}=0}^{q-1}\zeta^{x_{\ell}y_{\ell}}\right)\right).
\end{equation*}
Since $i>j$, there is an $\ell\ne \ell_1,\dots,\ell_j$ such that $y_{\ell}\ne 0$.
For this $\ell$, we have $\sum_{x_{\ell}=0}^{q-1}\zeta^{x_{\ell}y_{\ell}}=0$.
Thus, $f_z$ is orthogonal to $\varepsilon_y$.
It follows that $f_z$ is orthogonal to $\sum_{i=j+1}^dL_i(X)$, and hence it is contained in $\sum_{i=0}^jL_i(X)$, as desired.

\subsection*{Acknowledgements}
Eiichi Bannai is supported in part by NSFC grant No.~11271257.
Sho Suda was supported by JSPS Research Fellowships for Young Scientists.
Hajime Tanaka was supported in part by JSPS KAKENHI Grant No.~23740002 and No.~25400034.

\end{document}